\documentclass[12pt]{article}

\usepackage{amsmath,amsthm}
\usepackage{amssymb}
\usepackage{hyperref}
\usepackage{lscape}

\def\fl#1{\left\lfloor#1\right\rfloor}

\newtheorem{theorem}{Theorem}
\newtheorem{Prop}{Proposition}
\newtheorem{Cor}{Corollary}
\newtheorem{Lem}{Lemma}

\begin{document}

\title{Hypergeometric Euler numbers}

\author{
Takao Komatsu\\
\small School of Mathematics and Statistics\\
\small Wuhan University\\
\small Wuhan 430072 China\\
\small \texttt{komatsu@whu.edu.cn}\\\\
Huilin Zhu\\
\small School of Mathematical Sciences\\
\small Xiamen University\\
\small Xiamen 361005 China\\
\small \texttt{hlzhu@xmu.edu.cn}
}

\date{
\small MR Subject Classifications: 11B68, 11B37, 11C20, 15A15, 33C20.
}

\maketitle

\begin{abstract}
In this paper, we introduce the hypergeometric Euler number as an analogue of the hypergeometric Bernoulli number and the hypergeometric Cauchy number. We study several expressions and sums of products of hypergeometric Euler numbers.
We also introduce complementary hypergeometric Euler numbers and give some characteristic properties. There are strong reasons why these hypergeometric numbers are important. The hypergeometric numbers have one of the advantages that yield the natural extensions of determinant expressions of the numbers, though  many kinds of generalizations of the Euler numbers have been considered by many authors. 
\end{abstract}

\section{Introduction}

Euler numbers $E_n$ are defined by the generating function
\begin{equation}
\frac{1}{\cosh t}=\sum_{n=0}^\infty E_n\frac{t^n}{n!}\,.
\label{def:euler}
\end{equation}
One of the different definitions is
$$
\frac{2}{e^t+1}=\sum_{n=0}^\infty E_n\frac{t^n}{n!}
$$
(see e.g. \cite{Ap}). Generalizations of one or other of these definitions have previously been studied.  For example, one kind of poly-Euler numbers is a typical generalization, in the aspect of $L$-functions (\cite{OS1,OS2,OS3}). Other generalizations can be found in \cite{CCJK,KRS} and the reference therein.

A different type of generalization is based upon hypergeometric functions.
For $N\ge 1$, define hypergeometric Bernoulli numbers $B_{N,n}$ (see \cite{HN1,HN2,Kamano}) by
$$
\frac{1}{{}_1 F_1(1;N+1;t)}=\frac{t^N/N!}{e^t-\sum_{n=0}^{N-1}t^n/n!}=\sum_{n=0}^\infty B_{N,n}\frac{t^n}{n!}\,,
$$
where
$$
{}_1 F_1(a;b;z)=\sum_{n=0}^\infty\frac{(a)^{(n)}}{(b)^{(n)}}\frac{z^n}{n!}
$$
is the confluent hypergeometric function with $(x)^{(n)}=x(x+1)\cdots(x+n-1)$ ($n\ge 1$) and $(x)^{(0)}=1$.
When $N=1$, $B_n=B_{1,n}$ are classical Bernoulli numbers defined by
$$
\frac{t}{e^t-1}=\sum_{n=0}^\infty B_n\frac{t^n}{n!}\,.
$$
In addition, define hypergeometric Cauchy numbers $c_{N,n}$ (see \cite{Ko3}) by
$$
\frac{1}{{}_2 F_1(1,N;N+1;-t)}=\frac{(-1)^{N-1}t^N/N}{\log(1+t)-\sum_{n=1}^{N-1}(-1)^{n-1}t^n/n}=\sum_{n=0}^\infty c_{N,n}\frac{t^n}{n!}\,,
$$
where
$$
{}_2 F_1(a,b;c;z)=\sum_{n=0}^\infty\frac{(a)^{(n)}(b)^{(n)}}{(c)^{(n)}}\frac{z^n}{n!}
$$
is the Gauss hypergeometric function.
When $N=1$, $c_n=c_{1,n}$ are classical Cauchy numbers defined by
$$
\frac{t}{\log(1+t)}=\sum_{n=0}^\infty c_n\frac{t^n}{n!}\,.
$$
Our generalization is different from the generalizations in \cite{KLP} and the references therein.

Now, for $N\ge 0$ define {\it hypergeometric Euler numbers} $E_{N,n}$ ($n=0,1,2,\dots$) by
\begin{equation}
\frac{1}{{}_1 F_2(1;N+1,(2 N+1)/2;t^2/4)}=\sum_{n=0}^\infty E_{N,n}\frac{t^n}{n!}\,,
\label{def1:hypergeuler}
\end{equation}
where ${}_1 F_2(a;b,c;z)$ is the hypergeometric function defined by
$$
{}_1 F_2(a;b,c;z)=\sum_{n=0}^\infty\frac{(a)^{(n)}}{(b)^{(n)}(c)^{(n)}}\frac{z^n}{n!}\,.
$$
It is seen that
\begin{align}
\cosh t-\sum_{n=0}^{N-1}\frac{t^{2 n}}{(2 n)!}&=\frac{t^{2 N}}{(2 N)!}\sum_{n=0}^\infty\frac{(2 N)!n!}{(2 n+2 N)!}\frac{(t^2)^n}{n!}\notag\\
&=\frac{t^{2 N}}{(2 N)!}{}_1 F_2\bigl(1;N+1,\frac{2 N+1}{2};\frac{t^2}{4}\bigr)\,.
\label{def2:hypergeuler}
\end{align}
When $N=0$, then $E_n=E_{0,n}$ are classical Euler numbers defined in (\ref{def:euler}).
In \cite{KP}, the truncated Euler polynomial $E_{m,n}(x)$ is introduced as a generalization of the classical Euler polynomial $E_n(x)$. The concept is similar but without hypergeometric functions.

We list the numbers $E_{N,n}$ for $0\le N\le 6$ and $0\le n\le 12$ in Table 1.
From (\ref{def2:hypergeuler}) we see that $E_{N,n}=0$ if $n$ is odd.
Similarly to poly-Euler numbers (\cite{OS1,OS2,OS3}), hypergeometric Euler numbers are rational numbers, though the classical Euler numbers are integers.

\begin{landscape}
\begin{table}[phtb]
  \tiny
  \begin{center}
    \caption{The numbers $E_{N,n}$ for $0\le N\le 6$ and $0\le n\le 14$}
    \begin{tabular}{|l|cccccccc} \hline
    $n$&$0$&$2$&$4$&$6$&$8$&$10$&$12$&$14$ \\ \hline
    $E_{0,n}$&$1$&$-1$&$5$&$-61$&$1385$&$-50521$&$2702765$&$-199360981$\\
    $E_{1,n}$&$1$&$-1/6$&$1/10$&$-5/42$&$7/30$&$-15/22$&$7601/2730$&$-91/6$\\
    $E_{2,n}$&$1$&$-1/15$&$13/1050$&$-1/350$&$-31/173250$&$1343/750750$&$-6137/2388750$&$3499/6693750$\\
    $E_{3,n}$&$1$&$-1/28$&$17/5880$&$-29/362208$&$-863/6420960$&$6499/131843712$&$6997213/156894017280$&$-68936107/917226562560$\\
    $E_{4,n}$&$1$&$-1/45$&$7/7425$&$53/2027025$&$-443/22052250$&$-10157/4873547250$&$558599021 /126395447928750$&$39045649/62503243481250$\\
    $E_{5,n}$&$1$&$-1/66$&$25/66066$&$47/2906904$&$-16945/5300012718$&$-475767/492312292472$&$71844089/268802511689712$&$1162911301/4483980359834976$\\
    $E_{6,n}$&$1$&$-1/91$&$29/165620$&$1205/153728484$&$-2279/4467168888$&$-6430761/25339270989032$&$-17675104079/4917799642149532320$&$837165624457/24588998210747661600$\\ \hline
    \end{tabular}
  \end{center}
\end{table}
\end{landscape}

From (\ref{def1:hypergeuler}) and (\ref{def2:hypergeuler}), we have
\begin{align*}
\frac{t^{2 N}}{(2 N)!}&=\left(\sum_{n=N}^\infty\frac{t^{2 n}}{(2 n)!}\right)\left(\sum_{n=0}^\infty E_{N,n}\frac{t^n}{n!}\right)\\
&=t^{2 N}\left(\sum_{n=0}^\infty\frac{\frac{1+(-1)^n}{2}t^n}{(n+2 N)!}\right)\left(\sum_{n=0}^\infty E_{N,n}\frac{t^n}{n!}\right)\\
&=t^{2 N}\sum_{n=0}^\infty\left(\sum_{i=0}^n\frac{\frac{1+(-1)^{n-i}}{2}}{(2 N+n-i)!}\frac{E_{N,i}}{i!}\right)t^n\,.
\end{align*}
Hence, for $n\ge 1$, we have
$$
\sum_{i=0}^n\frac{1+(-1)^{n-i}}{(2 N+n-i)!i!}E_{N,i}=0\,.
$$
Thus, we have the following proposition.  Note that $E_{N,n}=0$ when $n$ is odd.

\begin{Prop}
$$
\sum_{i=0}^{n/2}\frac{1}{(2 N+n-2 i)!(2 i)!}E_{N,2 i}=0\quad\hbox{($n\ge 2$ is even)}
$$
and $E_{N,0}=1$.
\label{prop1}
\end{Prop}

By using the identity in Proposition \ref{prop1} or the identity
\begin{equation}
E_{N,n}=-n!(2 N)!\sum_{i=0}^{n/2-1}\frac{E_{N,2 i}}{(2 N+n-2 i)!(2 i)!}\,,
\label{rel:hgeuler}
\end{equation}
we can obtain the values of $E_{N,n}$ ($n=0,2,4,\dots$). We record the first few values of $E_{N,n}$:
{\small
\begin{align*}
E_{N,2}&=-\frac{2}{(2N+1)(2N+2)}\,,\\
E_{N,4}&=\frac{2\cdot 4!(4 N+5)}{(2 N+1)^2(2 N+2)^2(2 N+3)(2 N+4)}\,,\\
E_{N,6}&=\frac{4\cdot 6!(8 N^3-2 N^2-65 N-61)}{(2 N+1)^3(2 N+2)^3(2 N+3)(2 N+4)(2 N+5)(2 N+6)}\,,\\
E_{N,8}&=\frac{16\cdot 8!}{(2 N+1)^4(2 N+2)^4(2 N+3)^2(2 N+4)^2(2 N+6)(2 N+7)(2 N+8)}\\
&\quad\times(16 N^6-44 N^5-516 N^4-667 N^3+1283 N^2+3126 N+1662)\,.
\end{align*}
}

We have an explicit expression of $E_{N,n}$ for each even $n$:

\begin{theorem}
For $N\ge 0$ and $n\ge 1$ we have
$$
E_{N,2 n}=(2 n)!\sum_{r=1}^n(-1)^r\sum_{i_1+\cdots+i_r=n\atop i_1,\dots,i_r\ge 1}\frac{\bigl((2 N)!\bigr)^r}{(2 N+2 i_1)!\cdots(2 N+2 i_r)!}\,.
$$
\label{th1}
\end{theorem}

\begin{proof}
The proof is done by induction for $n$. If $n=1$, then
$$
E_{N,2}=2!(-1)\frac{(2 N)!}{(2 N+2)!}=-\frac{2}{(2N+1)(2N+2)}\,.
$$
Assume that the result is valid up to $n-1$. Then by Proposition \ref{prop1}
\begin{align*}
E_{N,2 n}&=-(2 n)!(2 N)!\sum_{i=0}^{n-1}\frac{E_{N,2 i}}{(2 N+2 n-2 i)!(2 i)!}\\
&=-(2 n)!(2 N)!\sum_{i=1}^{n-1}\frac{1}{(2 N+2 n-2 i)!}\sum_{r=1}^i(-1)^r\\
&\qquad\times\sum_{i_1+\cdots+i_r=i\atop i_1,\dots,i_r\ge 1}\frac{\bigl((2 N)!\bigr)^r}{(2 N+2 i_1)!\cdots(2 N+2 i_r)!}\\
&\quad-(2 n)!(2 N)!\frac{1}{(2 N+2 n)!}\\
&=-(2 n)!(2 N)!\sum_{r=1}^{n-1}(-1)^r\bigl((2 N)!\bigr)^r\sum_{i=r}^{n-1}\frac{1}{(2 N+2 n-2 i)!}\\
&\qquad\times\sum_{i_1+\cdots+i_r=i\atop i_1,\dots,i_r\ge 1}\frac{1}{(2 N+2 i_1)!\cdots(2 N+2 i_r)!}\\
&\quad-\frac{(2 n)!(2 N)!}{(2 N+2 n)!}\\
&=-(2 n)!(2 N)!\sum_{r=2}^{n}(-1)^{r-1}\bigl((2 N)!\bigr)^{r-1}\sum_{i=r-1}^{n-1}\frac{1}{(2 N+2 n-2 i)!}\\
&\qquad\times\sum_{i_1+\cdots+i_{r-1}=i\atop i_1,\dots,i_{r-1}\ge 1}\frac{1}{(2 N+2 i_1)!\cdots(2 N+2 i_{r-1})!}\\
 &\quad-\frac{(2 n)!(2 N)!}{(2 N+2 n)!}\\
 &=(2 n)!\sum_{r=2}^{n}(-1)^{r}\bigl((2 N)!\bigr)^{r}\sum_{i_1+\cdots+i_{r}=n\atop i_1,\dots,i_{r}\ge 1}\frac{1}{(2 N+2 i_1)!\cdots(2 N+2 i_r)!}\\
 &\quad-\frac{(2 n)!(2 N)!}{(2 N+2 n)!}\quad (n-i=i_r)\\
 &=(2 n)!\sum_{r=1}^n(-1)^r\sum_{i_1+\cdots+i_r=n\atop i_1,\dots,i_r\ge 1}\frac{\bigl((2 N)!\bigr)^r}{(2 N+2 i_1)!\cdots(2 N+2 i_r)!}\,.
 \end{align*}
 \end{proof}

\section{Determinant expressions of hypergeometric numbers}

These hypergeometric numbers have one of the advantages that yield the natural extensions of determinant expressions of the numbers, though  many kinds of generalizations of the Euler numbers have been considered by many authors.

By using Proposition \ref{prop1} or the relation (\ref{rel:hgeuler}), we have a determinant expression of hypergeometric Euler numbers (\cite{Ko12}).

\begin{Prop}
The hypergeometric Euler numbers $E_{N,2n}$ ($N\ge 0$, $n\ge 1$) can be expressed as
$$
E_{N, 2 n}=(-1)^n(2 n)!
\left|
\begin{array}{cccc}
\frac{(2 N)!}{(2 N+2)!}&1&&\\
\frac{(2 N)!}{(2 N+4)!}&\ddots&\ddots&\\
\vdots&&\ddots&1\\
\frac{(2 N)!}{(2 N+2 n)!}&\cdots&\frac{(2 N)!}{(2 N+4)!}&\frac{(2 N)!}{(2 N+2)!}
\end{array}
\right|\,.
$$
\label{det:hge}
\end{Prop}

In 1875, Glaisher gave several interesting determinant expressions of numbers, including Bernoulli, Cauchy and Euler numbers.
When $N=0$, the determinant in Proposition (\ref{det:hge}) is reduced to a famous determinant expression of Euler numbers ({\it cf.} \cite[p.52]{Glaisher}):
$$
E_{2n}=(-1)^n (2n)!
\begin{vmatrix}
\frac{1}{2!}& 1 &~& ~&~\\
\frac{1}{4!}&  \frac{1}{2!} & 1 &~&~\\
\vdots & ~  &  \ddots~~ &\ddots~~ & ~\\
\frac{1}{(2n-2)!}& \frac{1}{(2n-4)!}& ~&\frac{1}{2!} &  1\\
\frac{1}{(2n)!}&\frac{1}{(2n-2)!}& \cdots &  \frac{1}{4!} & \frac{1}{2!}
\end{vmatrix}\,.
$$

In \cite{AK1}, the hypergeometric Bernoulli numbers $B_{N,n}$ ($N\ge 1$, $n\ge 1$) can be expressed as
$$
B_{N,n}=(-1)^n n!\left|
\begin{array}{ccccc}
\frac{N!}{(N+1)!}&1&&&\\
\frac{N!}{(N+2)!}&\frac{N!}{(N+1)!}&&&\\
\vdots&\vdots&\ddots&1&\\
\frac{N!}{(N+n-1)!}&\frac{N!}{(N+n-2)!}&\cdots&\frac{N!}{(N+1)!}&1\\
\frac{N!}{(N+n)!}&\frac{N!}{(N+n-1)!}&\cdots&\frac{N!}{(N+2)!}&\frac{N!}{(N+1)!}
\end{array}
\right|\,.
$$
When $N=1$, we have a determinant expression of Bernoulli numbers (\cite[p.53]{Glaisher}):
\begin{equation}
B_n=(-1)^n n!\left|
\begin{array}{ccccc}
\frac{1}{2!}&1&&&\\
\frac{1}{3!}&\frac{1}{2!}&&&\\
\vdots&\vdots&\ddots&1&\\
\frac{1}{n!}&\frac{1}{(n-1)!}&\cdots&\frac{1}{2!}&1\\
\frac{1}{(n+1)!}&\frac{1}{n!}&\cdots&\frac{1}{3!}&\frac{1}{2!}
\end{array}
\right|\,.
\label{det:ber}
\end{equation}

In \cite{AK2}, the hypergeometric Cauchy numbers $c_{N,n}$ ($N\ge 1$, $n\ge 1$) can be expressed as
$$
c_{N,n}=n!\left|
\begin{array}{ccccc}
\frac{N}{N+1}&1&&&\\
\frac{N}{N+2}&\frac{N}{N+1}&&&\\
\vdots&\vdots&\ddots&1&\\
\frac{N}{N+n-1}&\frac{N}{N+n-2}&\cdots&\frac{N}{N+1}&1\\
\frac{N}{N+n}&\frac{N}{N+n-1}&\cdots&\frac{N}{N+2}&\frac{N}{N+1}
\end{array}
\right|\,.
$$
When $N=1$, we have a determinant expression of Cauchy numbers (\cite[p.50]{Glaisher}):
\begin{equation}
c_n=n!\left|
\begin{array}{ccccc}
\frac{1}{2}&1&&&\\
\frac{1}{3}&\frac{1}{2}&&&\\
\vdots&\vdots&\ddots&1&\\
\frac{1}{n}&\frac{1}{n-1}&\cdots&\frac{1}{2}&1\\
\frac{1}{n+1}&\frac{1}{n}&\cdots&\frac{1}{3}&\frac{1}{2}
\end{array}
\right|\,.
\label{det:cau}
\end{equation}

In \cite{Ko12}, the complementary Euler numbers $\widehat E_n$ and their hypergeometric generalizations (defined below) have also determinant expressions.

\section{Hasse-Teichm\"uller derivative}

We define the Hasse-Teichm\"uller derivative $H^{(n)}$ of order $n$ by
$$
H^{(n)}\left(\sum_{m=R}^{\infty} c_m z^m\right)
=\sum_{m=R}^{\infty} c_m \binom{m}{n}z^{m-n}
$$
for $\sum_{m=R}^{\infty} c_m z^m\in \mathbb{F}((z))$,
where $R$ is an integer and $c_m\in\mathbb{F}$ for any $m\geq R$.

The Hasse-Teichm\"uller derivatives satisfy the product rule \cite{Teich}, the quotient rule \cite{GN} and the chain rule \cite{Hasse}.
One of the product rules can be described as follows.
\begin{Lem}
For $f_i\in\mathbb F[[z]]$ ($i=1,\dots,k$) with $k\ge 2$ and for $n\ge 1$, we have
$$
H^{(n)}(f_1\cdots f_k)=\sum_{i_1,\dots,i_k\ge 0\atop i_1+\cdots+i_k=n}H^{(i_1)}(f_1)\cdots H^{(i_k)}(f_k)\,.
$$
\label{productrule2}
\end{Lem}

The quotient rules can be described as follows.

\begin{Lem}
For $f\in\mathbb F[[z]]\backslash \{0\}$ and $n\ge 1$,
we have
\begin{align}
H^{(n)}\left(\frac{1}{f}\right)&=\sum_{k=1}^n\frac{(-1)^k}{f^{k+1}}\sum_{i_1,\dots,i_k\ge 1\atop i_1+\cdots+i_k=n}H^{(i_1)}(f)\cdots H^{(i_k)}(f)
\label{quotientrule1}\\
&=\sum_{k=1}^n\binom{n+1}{k+1}\frac{(-1)^k}{f^{k+1}}\sum_{i_1,\dots,i_k\ge 0\atop i_1+\cdots+i_k=n}H^{(i_1)}(f)\cdots H^{(i_k)}(f)\,.
\label{quotientrule2}
\end{align}
\label{quotientrules}
\end{Lem}

By using the Hasse-Teichm\"uller derivative of order $n$, we shall obtain some explicit expressions of the hypergeometric Euler numbers.
\bigskip

{\it Another proof of Theorem \ref{th1}.}
Put
\begin{align*}
F:&={}_1 F_2\bigl(1;N+1,\frac{2 N+1}{2};\frac{t^2}{4}\bigr)\\
&=\sum_{n=0}^\infty\frac{(2 N)!}{(2 N+2 n)!}t^{2 n}
\end{align*}
for simplicity.
Note that
\begin{align*}
\left. H^{(i)}(F)\right|_{t=0}=\left.\sum_{j=0}^\infty\frac{(2 N)!}{(2 N+2 j)!}\binom{2 j}{i}t^{2 j-i}\right|_{t=0}
=
\begin{cases}
(2 N)!/(2 N+i)!&\text{if $i$ is even};\\
0&\text{if $i$ is odd}.
\end{cases}
\end{align*}
Hence, by using Lemma \ref{quotientrules} (\ref{quotientrule1}), we have
\begin{align*}
\frac{E_{N,n}}{n!}&=\left.H^{(n)}\left(\frac{1}{F}\right)\right|_{t=0}\\
&=\sum_{k=1}^n\left.\frac{(-1)^k}{F^{k+1}}\right|_{t=0}\sum_{i_1,\dots,i_k\ge 1\atop i_1+\cdots+i_k=n}\left.H^{(i_1)}(F)\right|_{t=0}\cdots\left.H^{(i_k)}(F)\right|_{t=0}\\
&=\sum_{k=1}^n(-1)^k\sum_{i_1,\dots,i_k\ge 1\atop 2(i_1+\cdots+i_k)=n}\frac{\bigl((2 N)!\bigr)^k}{(2 N+2 i_1)!\cdots(2 N+2 i_k)!}\,.
\end{align*}
\qed

We can express the hypergeometric Euler numbers also in terms of the binomial coefficients.
In fact, by using Lemma \ref{quotientrules} (\ref{quotientrule2})
instead of Lemma \ref{quotientrules} (\ref{quotientrule1})
in the above proof, we obtain a little different expression from one in Theorem \ref{th1}.

\begin{Prop}
For $N\ge 0$ and even $n\ge 2$,
$$
E_{N,n}=n!\sum_{k=1}^{n}(-1)^k\binom{n+1}{k+1}\sum_{i_1,\dots,i_k\ge 0\atop i_1+\cdots+i_k=n/2}\frac{\bigl((2 N)!\bigr)^k}{(2 N+2 i_1)!\cdots(2 N+2 i_k)!}\,.
$$
\label{th_euler2}
\end{Prop}

For example, when $n=4$, we get
\begin{align*}
E_4&=4!\left(-\binom{5}{2}\frac{1}{4!}+\binom{5}{3}\left(\frac{2}{4!}+\frac{1}{2!2!}\right)\right.\\
&\qquad\left.-\binom{5}{4}\left(\frac{3}{4!}+\frac{3}{2!2!}\right)+\binom{5}{5}\left(\frac{4}{4!}+\frac{6}{2!2!}\right)\right)\\
&=5\,,\\
E_{1,4}&=4!\left(-\binom{5}{2}2\frac{1}{6!}+\binom{5}{3}2^2\left(\frac{2}{6!2!}+\frac{1}{4!4!}\right)\right.\\
&\qquad\left.-\binom{5}{4}2^3\left(\frac{3}{6!2!2!}+\frac{3}{4!4!2!}\right)+\binom{5}{5}2^4\left(\frac{4}{6!2!2!2!}+\frac{6}{4!4!2!2!}\right)\right)\\
&=\frac{1}{10}\,,\\
E_{2,4}&=4!\left(-\binom{5}{2}4!\frac{1}{8!}+\binom{5}{3}(4!)^2\left(\frac{2}{8!4!}+\frac{1}{6!6!}\right)\right.\\
&\qquad\left.-\binom{5}{4}(4!)^3\left(\frac{3}{8!4!4!}+\frac{3}{6!6!4!}\right)+\binom{5}{5}(4!)^4\left(\frac{4}{8!4!4!4!}+\frac{6}{6!6!4!4!}\right)\right)\\
&=\frac{13}{1050}\,,\\
E_{3,4}&=4!\left(-\binom{5}{2}6!\frac{1}{10!}+\binom{5}{3}(6!)^2\left(\frac{2}{10!6!}+\frac{1}{8!8!}\right)\right.\\
&\qquad\left.-\binom{5}{4}(6!)^3\left(\frac{3}{10!6!6!}+\frac{3}{8!8!6!}\right)+\binom{5}{5}(6!)^4\left(\frac{4}{10!6!6!6!}+\frac{6}{8!8!6!6!}\right)\right)\\
&=\frac{17}{5880}\,.\\
\end{align*}
\bigskip

\section{Some hypergeometric Euler numbers}

If $N=1$, we have the following relation between hypergeometric Euler numbers and Bernoulli numbers.

\begin{theorem}
For $n\ge 1$ we have
$$
E_{1,n}=-(n-1)B_n\,.
$$
\label{th22}
\end{theorem}
\begin{proof}
The result is clear for $n=0,1$ and odd numbers $n$.
By using the following Lemma \ref{lem11} and Proposition \ref{prop1}, we get the result.
\end{proof}

\begin{Lem}
For $n\ge 1$ we have
$$
\sum_{i=0}^n\frac{(i-1)B_i}{(n-i+2)!i!}=
\begin{cases}
0&\text{if $n$ is even};\\
-B_{n+1}/n!&\text{if $n$ is odd}.
\end{cases}
$$
\label{lem11}
\end{Lem}
\begin{proof}
Firstly,
\begin{align*}
&\sum_{n=0}^\infty\sum_{i=0}^n\frac{(i-1)B_i}{(n-i+2)!i!}x^n\\
&=\left(\sum_{k=0}^\infty\frac{x^k}{(k+2)!}\right)\left(\sum_{i=0}^\infty(i-1)B_i\frac{x^i}{i!}\right)\\
&=\left(\frac{1}{x^2}\sum_{k=0}^\infty\frac{x^{k+2}}{(k+2)!}\right)\left(-2\sum_{i=0}^\infty B_i\frac{x^i}{i!}+\frac{d}{d x}\sum_{i=0}^\infty B_i\frac{x^{i+1}}{i!}\right)\\
&=\frac{e^x-1-x}{x^2}\left(-\frac{2 x}{e^x-1}+\frac{2 x(e^x-1)-x^2 e^x}{(e^x-1)^2}\right)\\
&=\frac{e^x(x+1-e^x)}{(e^x-1)^2}\,.
\end{align*}
On the other hand,
\begin{align*}
&-\frac{1}{2}-\sum_{n=0}^\infty B_{2 n+2}\frac{x^{2 n+1}}{(2 n+1)!}\\
&=-\frac{1}{2}-\frac{d}{d x}\left(\sum_{n=0}^\infty B_n\frac{x^n}{n!}-B_0-B_1 x\right)\\
&=-\frac{1}{2}-\frac{d}{d x}\left(\frac{x}{e^x-1}-1+\frac{x}{2}\right)\\
&=\frac{e^x(x+1-e^x)}{(e^x-1)^2}\,.
\end{align*}
Comparing the coefficients of $x^n$, we get the result.
\end{proof}

\section{Sums of products of hypergeometric Euler numbers}

It is known that
$$
\sum_{i=0}^n\binom{2 n}{2 i}E_{2 i}=0
$$
with $E_0=1$, and $E_{2 i-1}=0$ ($i\ge 1$).

First, let us consider the sums of products of hypergeometric Euler numbers:
$$
Y_{N, 2}(n)=\sum_{i=0}^n\binom{2 n}{2 i}E_{N,2 i}E_{N, 2 n-2 i}\,.
$$
It is clear that
$$
\sum_{i=0}^n\binom{n}{i}E_{N,i}E_{N, n-i}=0
$$
if $n$ is odd.

If $N=0$, then
$$
Y_{0,2}(n)=\frac{2^{2 n+2}(2^{2 n+2}-1)B_{2 n+2}}{2 n+2}\quad(n\ge 0)\,.
$$
Indeed,
$$
\{Y_{0,2}(n)\}_{n\ge 0}=1,-2,16, -272, 7936, -353792, 22368256, -1903757312, \dots.
$$
The numbers taking their absolute value are called the {\it tangent numbers} or the {\it zag numbers} (\cite[A000182]{oeis}).
Thus, we also have
$$
Y_{0,2}(n)=\sum_{k=1}^{2 n+2}\sum_{j=0}^k\binom{k}{j}\frac{(-1)^{j+1}(k-2 j)^{2 n+2}}{2^k\sqrt{-1}^k k}\,.
$$
In other words, they appear as numerators in the Maclaurin series of $\tan x$:
$$
\tan x=\sum_{n=0}^\infty(-1)^n Y_{0,2}(n)\frac{x^{2 n+1}}{( 2n+1)!}\,.
$$

Put
\begin{align*}
F:&={}_1 F_2\bigl(1;N+1,\frac{2 N+1}{2};\frac{t^2}{4}\bigr)\\
&=\sum_{n=0}^\infty\frac{(2 N)!}{(2 N+2 n)!}t^{2 n}
\end{align*}
for simplicity again.
Then by
$$
\frac{d}{d t}F=\sum_{n=0}^\infty\frac{(2 n)(2 N)!}{(2 N+2 n)!}t^{2 n-1}\,,
$$
we have
\begin{equation}
2 N F+t\frac{t}{d t}F=2 N\cdot{}_1 F_2\bigl(1;N,\frac{2 N+1}{2};\frac{t^2}{4}\bigr)\,.
\label{eq401}
\end{equation}
For further simplicity, we put for $k=1,2,\dots,2 N$
$$
F_{(2 N-k)}={}_1 F_2\left(1;\fl{\frac{k+2}{2}},\fl{\frac{k+1}{2}}+\frac{1}{2};\frac{t^2}{4}\right)
$$
with $F_{(0)}=F$.
Then,
in general, we obtain for $k=1,2,\dots,2 N$
\begin{equation}
k F_{(2 N-k)}+t\frac{d}{d t}F_{(2 N-k)}=k F_{(2 N-k+1)}\,.
\label{eq50}
\end{equation}

\begin{Prop}
For $k=0,1,\dots,2 N$ we have
$$
\cosh t=\sum_{i=0}^k\frac{t^i}{i!}\binom{k}{i}\frac{d^i}{d t^i}F_{(2 N-k)}\,.
$$
\end{Prop}
\begin{proof}
For $k=0$, we get
$$
F_{(2 N)}=\sum_{n=0}^\infty\frac{t^{2 n}}{(2 n)!}=\cosh t\,.
$$
Assume that the result holds for some $k\ge 0$.  Then by (\ref{eq50})
\begin{align*}
&\sum_{i=0}^k\frac{t^i}{i!}\binom{k}{i}\frac{d^i}{d t^i}F_{(2 N-k)}\\
&=\sum_{i=0}^k\frac{t^i}{i!}\binom{k}{i}\frac{d^i}{d t^i}\left(F_{(2 N-k-1)}+\frac{t}{k+1}\frac{d}{d t}F_{(2 N-k-1)}\right)\\
&=\sum_{i=0}^k\frac{t^i}{i!}\binom{k}{i}\\
&\qquad\times\left(\frac{d^i}{d t^i}F_{(2 N-k-1)}+\frac{i}{k+1}\frac{d^i}{d t^i}F_{(2 N-k-1)}+\frac{t}{k+1}\frac{d^{i+1}}{d t^{i+1}}F_{(2 N-k-1)}\right)\\
&=\sum_{i=0}^k\frac{t^i}{i!}\binom{k}{i}\frac{k+i+1}{k+1}\frac{d^i}{d t^i}F_{(2 N-k-1)}\\
&\qquad +\sum_{i=1}^{k+1}\frac{t^{i-1}}{(i-1)!}\binom{k}{i-1}\frac{t}{k+1}\frac{d^i}{d t^i}F_{(2 N-k-1)}\\
&=\sum_{i=0}^{k+1}\frac{t^i}{i!}\binom{k+1}{i}\frac{d^i}{d t^i}F_{(2 N-k-1)}\,.
\end{align*}
\end{proof}

We introduce the {\it complementary hypergeometric Euler numbers} $\widehat E_{N,n}$ by
$$
\frac{t^{2N+1}/(2 N+1)!}{\sinh t-\sum_{n=0}^{N-1}t^{2 n+1}/(2 n+1)!}=\sum_{n=0}^\infty\widehat E_{N,n}\frac{t^n}{n!}
$$
as an analogue of (\ref{def1:hypergeuler}). When $n=0$,
$\widehat E_n=\widehat E_{0,n}$ are the {\it complementary Euler numbers} defined by
$$
\frac{t}{\sinh t}=\sum_{n=0}^\infty\widehat E_n\frac{t^n}{n!}
$$
as an analogue of (\ref{def:euler}).  In \cite{KP}, they are called {\it weighted Bernoulli numbers}, but this naming means different in other literatures.
Since
\begin{align*}
F^\ast:&={}_1 F_2\bigl(1;N,\frac{2 N+1}{2};\frac{t^2}{4}\bigr)\\
&=\sum_{n=0}^\infty\frac{(2 N-1)!}{(2 N+2 n-1)!}t^{2 n}
\end{align*}
and
\begin{equation}
\frac{d}{d t}F=-F^2\frac{d}{d t}\frac{1}{F}\,,
\label{eq44}
\end{equation}
by (\ref{eq401})
we have
\begin{equation}
\frac{1}{F^2}=\frac{1}{F^\ast}\left(\frac{1}{F}-\frac{t}{2 N}\frac{d}{d t}\frac{1}{F}\right)\,.
\label{eq402}
\end{equation}
Since
\begin{align*}
\frac{1}{F^\ast}&=\frac{t^{2 N-1}}{(2 N-1)!\sum_{n=0}^\infty t^{2 N+2 n-1}/(2 N+ 2 n-1)!}\\
&=\sum_{n=0}^\infty\widehat E_{N-1,n}\frac{t^n}{n!}
\end{align*}
and
\begin{align*}
\frac{1}{F}-\frac{t}{2 N}\frac{d}{d t}\frac{1}{F}
&=\sum_{n=0}^\infty E_{N,n}\frac{t^n}{n!}-\frac{t}{2 N}\sum_{n=1}^\infty E_{N,n}\frac{t^{n-1}}{(n-1)!}\\
&=\sum_{n=0}^\infty\frac{2 N-n}{2 N}E_{N,n}\frac{t^n}{n!}\,,
\end{align*}
we have
\begin{align*}
\frac{1}{F^\ast}\left(\frac{1}{F}-\frac{t}{2 N}\frac{d}{d t}\frac{1}{F}\right)
&=\left(\sum_{m=0}^\infty\widehat E_{N-1,m}\frac{t^m}{m!}\right)\left(\sum_{k=0}^\infty\frac{2 N-k}{2 N}E_{N,k}\frac{t^k}{k!}\right)\\
&=\sum_{n=0}^\infty\sum_{k=0}^n\binom{n}{k}\frac{2 N-k}{2 N}E_{N,k}\widehat E_{N-1,n-k}\frac{t^n}{n!}\,.
\end{align*}
Comparing the coefficients, we obtain a result about the sums of products.

\begin{theorem}
For $N\ge 1$ and $n\ge 0$,
$$
\sum_{i=0}^n\binom{n}{i}E_{N,i}E_{N,n-i}
=\sum_{k=0}^n\binom{n}{k}\frac{2 N-k}{2 N}E_{N,k}\widehat E_{N-1,n-k}\,.
$$
\label{th:sumprod2}
\end{theorem}

Using (\ref{eq44}) and (\ref{eq402}) again,
we have
\begin{align*}
\frac{1}{F^3}&=\frac{1}{F^\ast}\left(\frac{1}{F^2}-\frac{t}{2 N}\frac{1}{F}\frac{d}{d t}\frac{1}{F}\right)\\
&=\frac{1}{F^\ast}\left(\frac{1}{F^2}-\frac{t}{4 N}\frac{d}{d t}\frac{1}{F^2}\right)\,.
\end{align*}
Since
$$
\frac{1}{F^2}-\frac{t}{4 N}\frac{d}{d t}\frac{1}{F^2}
=\sum_{n=0}^\infty\frac{4 N-n}{4 N}\sum_{k=0}^n\binom{n}{k}\frac{2 N-k}{2 N}E_{N,k}\widehat E_{N-1,n-k}\frac{t^n}{n!}\,,
$$
we have
\begin{align*}
&\frac{1}{F^\ast}\left(\frac{1}{F^2}-\frac{t}{4 N}\frac{d}{d t}\frac{1}{F^2}\right)\\
&=\left(\sum_{i=0}^\infty\widehat E_{N-1,i}\frac{t^i}{i!}\right)\left(\sum_{m=0}^\infty\frac{4 N-m}{4 N}\sum_{k=0}^m\binom{m}{k}\frac{2 N-k}{2 N}E_{N,k}\widehat E_{N-1,m-k}\frac{t^m}{m!}\right)\\
&=\sum_{n=0}^\infty\sum_{m=0}^n\sum_{k=0}^m\binom{n}{m}\binom{m}{k}\frac{(4 N-m)(2 N-k)}{8 N^2}E_{N,k}\widehat E_{N-1,n-m}\widehat E_{N-1,m-k}\frac{t^n}{n!}\,,
\end{align*}
Comparing the coefficients, we get a result about the sums of products for trinomial coefficients.

\begin{theorem}
For $N\ge 1$ and $n\ge 0$,
\begin{multline*}
\sum_{i_1+i_2+i_3=n\atop i_1,i_2,i_3\ge 0}\binom{n}{i_1, i_2, i_3}E_{N,i_1}E_{N,i_2}E_{N,i_3}\\
=\sum_{m=0}^n\sum_{k=0}^m\binom{n}{m}\binom{m}{k}\frac{(4 N-m)(2 N-k)}{8 N^2}E_{N,k}\widehat E_{N-1,n-m}\widehat E_{N-1,m-k}\,.
\end{multline*}
\label{th:sumprod3}
\end{theorem}

\subsection{Complementary hypergeometric Euler numbers}

By using the similar methods in previous sections, the complementary hypergeometric Euler numbers satisfy the recurrence relation for even $n$
$$
\sum_{i=0}^{n/2}\frac{\widehat E_{N,2 i}}{(2 N+n-2 i+1)!(2 i)!}=0
$$
or
$$
\widehat E_{N, n}=-n!(2 N+1)!\sum_{i=0}^{n/2-1}\frac{\widehat E_{N,2 i}}{(2 N+n-2 i+1)!(2 i)!}\,.
$$

By using the Hasse-Teichm\"uller derivative or by proving by induction, we have

\begin{theorem}
For $N\ge 0$ and $n\ge 1$ we have
\begin{align*}
\widehat E_{N,n}&=n!\sum_{k=1}^n(-1)^k\sum_{i_1,\dots,i_k\ge 1\atop i_1+\cdots+i_k=n/2}\frac{\bigl((2 N+1)!\bigr)^k}{(2 N+2 i_1+1)!\cdots(2 N+2 i_k+1)!}\\
&=n!\sum_{k=1}^n(-1)^k\binom{n+1}{k+1}\sum_{i_1,\dots,i_k\ge 0\atop i_1+\cdots+i_k=n/2}\frac{\bigl((2 N+1)!\bigr)^k}{(2 N+2 i_1+1)!\cdots(2 N+2 i_k+1)!}\,.
\end{align*}
\end{theorem}

Some initial values of $\widehat E_{N,n}$ ($n=0,2,4,\dots$), we have
{\small
\begin{align*}
\widehat E_{N,2}&=-\frac{2}{(2N+2)(2N+3)}\,,\\
\widehat E_{N,4}&=\frac{2\cdot 4!(4 N+7)}{(2 N+2)^2(2 N+3)^2(2 N+4)(2 N+5)}\,,\\
\widehat E_{N,6}&=\frac{4\cdot 6!(8 N^3+10 N^2-61 N-93))}{(2 N+2)^3(2 N+3)^3(2 N+4)(2 N+5)(2 N+6)(2 N+7)}\,,\\
\widehat E_{N,8}&=\frac{8\cdot 8!}{(2 N+2)^4(2 N+3)^4(2 N+4)^2(2 N+5)^2(2 N+7)(2 N+8)(2 N+9)}\\
&\quad\times(32 N^6+8 N^5-1132 N^4-3538 N^3-1063 N^2+7280 N+6858)\,.
\end{align*}

Put
$$
\widehat F=\sum_{n=0}^\infty\frac{(2 N+1)!}{(2 N+2 n+1)!}t^{2 n}
$$
so that
$$
\frac{1}{\widehat F}=\sum_{n=0}^\infty\widehat E_{N,n}\frac{t^n}{n!}\,.
$$
Since
$$
(2 N+1)\widehat F+t\frac{d}{d t}\widehat F=(2 N+1)F\,,
$$
we have
\begin{align*}
\frac{1}{\widehat F^2}&=\frac{1}{F}\left(\frac{1}{\widehat F}-\frac{t}{2 N+1}\frac{d}{d t}\frac{1}{\widehat F}\right)\\
&=\left(\sum_{m=0}^\infty E_{N,m}\frac{t^m}{m!}\right)\left(\sum_{k=0}^\infty\frac{2 N-k+1}{2 N+1}\widehat E_{N,k}\frac{t^k}{k!}\right)\\
&=\sum_{n=0}^\infty\sum_{k=0}^n\binom{n}{k}\frac{2 N-k+1}{2 N+1}\widehat E_{N,k}E_{N,n-k}\frac{t^n}{n!}\,.
\end{align*}

Hence, as an analogue of Theorem \ref{th:sumprod2},  we have the following.

\begin{theorem}
For $N\ge 1$ and $n\ge 0$,
$$
\sum_{i=0}^n\binom{n}{i}\widehat E_{N,i}\widehat E_{N,n-i}
=\sum_{k=0}^n\binom{n}{k}\frac{2 N-k+1}{2 N+1}\widehat E_{N,k}E_{N,n-k}\,.
$$
\label{th:sumprod2a}
\end{theorem}

We then have
$$
\frac{1}{\widehat F^3}=\frac{1}{F}\left(\frac{1}{\widehat F^2}-\frac{t}{2(2 N+1)}\frac{d}{d t}\frac{1}{\widehat F^2}\right)\,.
$$
Since
$$
\frac{1}{\widehat F^2}-\frac{t}{2(2 N+1)}\frac{d}{d t}\frac{1}{\widehat F^2}
=\sum_{n=0}^\infty\frac{4 N-n+2}{2(2 N+1)}\sum_{k=0}^n\binom{n}{k}\frac{2 N-k+1}{2 N+1}\widehat E_{N,k}E_{N,n-k}\frac{t^n}{n!}\,,
$$
we have the following result as an analogue of Theorem \ref{th:sumprod3}.

\begin{theorem}
For $N\ge 1$ and $n\ge 0$,
\begin{multline*}
\sum_{i_1+i_2+i_3=n\atop i_1,i_2,i_3\ge 0}\binom{n}{i_1, i_2, i_3}\widehat E_{N,i_1}\widehat E_{N,i_2}\widehat E_{N,i_3}\\
=\sum_{m=0}^n\sum_{k=0}^m\binom{n}{m}\binom{m}{k}\frac{(4 N-m+2)(2 N-k+1)}{2(2 N+1)^2}\widehat E_{N,k}E_{N,n-m}E_{N,m-k}\,.
\end{multline*}
\label{th:sumprod3a}
\end{theorem}

One can continue to obtain the sum of four or more products, though the results seem to become more complicated.

\section{Applications by the Trudi's formula}

We shall use the Trudi's formula to obtain different explicit expressions for the hypergeometric Euler numbers $E_{N,n}$.

\begin{Lem}[Trudi's formula \cite{KR,trudi}]
For a positive integer $m$, we have
\begin{multline*}
\left|
\begin{array}{ccccc}
a_1  & a_2   &  \cdots   & & a_m  \\
a_{0}  & a_{1}    &  \cdots  & &   \\
\vdots  &  \vdots &  \ddots  &  & \vdots  \\
0  & 0    &  \cdots  &a_1  & a_{2}  \\
0  & 0   &  \cdots  & a_0  & a_1
\end{array}
\right|\\
=
\sum_{t_1 + 2t_2 + \cdots +mt_m=m}\binom{t_1+\cdots + t_m}{t_1, \dots,t_m}(-a_0)^{m-t_1-\cdots - t_m}a_1^{t_1}a_2^{t_2}\cdots a_m^{t_m}, \label{trudi}
\end{multline*}
where $\binom{t_1+\cdots + t_m}{t_1, \dots,t_m}=\frac{(t_1+\cdots + t_m)!}{t_1 !\cdots t_m !}$ are the multinomial coefficients.
\label{lema0}
\end{Lem}

This relation is known as Trudi's formula \cite[Vol.3, p.214]{Muir},\cite{Trudi} and the case $a_0=1$ of this formula is known as Brioschi's formula \cite{Brioschi},\cite[Vol.3, pp.208--209]{Muir}.

In addition, there exists the following inversion formula (see, e.g. \cite{KR}), which is based upon the relation:
$$
\sum_{k=0}^n(-1)^{n-k}\alpha_k D(n-k)=0\quad(n\ge 1)\,.
$$

\begin{Lem}
If $\{\alpha_n\}_{n\geq 0}$ is a sequence defined by $\alpha_0=1$ and
$$
\alpha_n=\begin{vmatrix} R(1) & 1 & & \\
R(2) & \ddots &  \ddots & \\
\vdots & \ddots &  \ddots & 1\\
R(n) & \cdots &  R(2) & R(1) \\
 \end{vmatrix},  \ \text{then} \ R(n)=\begin{vmatrix} \alpha_1 & 1 & & \\
\alpha_2 & \ddots &  \ddots & \\
\vdots & \ddots &  \ddots & 1\\
\alpha_n & \cdots &  \alpha_2 & \alpha_1 \\
 \end{vmatrix}\,.
$$
Moreover, if
$$
A=\begin{pmatrix}
  1 &  & & \\
\alpha_1 & 1  &   & \\
\vdots & \ddots &  \ddots & \\
\alpha_n& \cdots &  \alpha_1 & 1 \\
 \end{pmatrix}, \ \text{then} \  A^{-1}=\begin{pmatrix}
  1 &  & & \\
R(1) & 1  &   & \\
\vdots & \ddots &  \ddots & \\
R(n) & \cdots &  R(1) & 1 \\
 \end{pmatrix}\,.
$$
\label{lema}
\end{Lem}

From Trudi's formula,  it is possible to give the combinatorial expression
$$
\alpha_n=\sum_{t_1+2t_2+\cdots +n t_n=n}\binom{t_1+\cdots+t_n}{t_1, \dots, t_n}(-1)^{n-t_1-\cdots - t_n}R(1)^{t_1}R(2)^{t_2}\cdots R(n)^{t_n}\,.
$$
By applying these lemmata to Proposition \ref{det:hge}, we obtain an explicit expression for the hypergeometric Euler numbers $E_{N,n}$.

\begin{theorem}
For $N\ge 0$ and $n\geq 1$,
\begin{multline*}
E_{N,2 n}
=(2 n)!\sum_{t_1 + 2t_2 + \cdots + nt_n=n}\binom{t_1+\cdots + t_n}{t_1, \dots,t_n}(-1)^{t_1+\cdots+t_n}\\
\times\left(\frac{(2 N)!}{(2 N+2)!}\right)^{t_1}\left(\frac{(2 N)!}{(2 N+4)!}\right)^{t_2}\cdots\left(\frac{(2 N)!}{(2 N+2 n)!}\right)^{t_n}\,.
\end{multline*}
Moreover,
$$
\frac{(-1)^n(2 N)!}{(2 N+2 n)!}=\begin{vmatrix} \frac{E_{N,2}}{2!} & 1 & & \\
\frac{E_{N,4}}{4!} & \ddots &  \ddots & \\
\vdots & \ddots &  \ddots & 1\\
\frac{E_{N, 2 n}}{(2 n)!} & \cdots &  \frac{E_{N,4}}{4!} & \frac{E_{N,2}}{2!} \\
 \end{vmatrix}\,,
$$
and
{\small
\begin{multline*}
\begin{pmatrix} 1 &  &  & & \\
-\frac{E_{N,2}}{2!} & 1 &   &  & \\
\frac{E_{N,4}}{4!} & -\frac{E_{N,2}}{2!}  &  1  &  & \\
\vdots &  &  \ddots &  & \\
\frac{(-1)^n E_{N,2 n}}{(2 n)!} & \cdots &  \frac{E_{N,4}}{4!} & -\frac{E_{N,2}}{2!} &1
\end{pmatrix}^{-1}\\
=\begin{pmatrix} 1 &  &  & & \\
\frac{(2 N)!}{(2 N+2)!} & 1 &   &  & \\
\frac{(2 N)!}{(2 N+4)!} & \frac{(2 N)!}{(2 N+2)!}  &  1  &  & \\
\vdots &  &  \ddots &  & \\
\frac{(2 N)!}{(2 N+2 n)!} & \cdots &  \frac{(2 N)!}{(2 N+4)!} & \frac{(2 N)!}{(2 N+2)!} &1
\end{pmatrix}\,.
\end{multline*}
}
\label{th1234}
\end{theorem}

When $N=0$ in Theorem \ref{th1234}, we have a different expression for the classical Euler numbers $E_{n}$.

\begin{Cor}
For $n\geq 1$
\begin{multline*}
E_{2 n}
=(2 n)!\sum_{t_1 + 2t_2 + \cdots + nt_n=n}\binom{t_1+\cdots + t_n}{t_1, \dots,t_n}(-1)^{t_1+\cdots+t_n}\\
\times\left(\frac{1}{2!}\right)^{t_1}\left(\frac{1}{4!}\right)^{t_2}\cdots \left(\frac{1}{(2 n)!}\right)^{t_n}\,.
\end{multline*}
Moreover,
$$
\frac{(-1)^n}{(2 n)!}=\begin{vmatrix} \frac{E_{2}}{2!} & 1 & & \\
\frac{E_{4}}{4!} & \ddots &  \ddots & \\
\vdots & \ddots &  \ddots & 1\\
\frac{E_{2 n}}{(2 n)!} & \cdots &  \frac{E_{4}}{4!} & \frac{E_{2}}{2!} \\
 \end{vmatrix}\,.
$$
\end{Cor}

Similarly, by the results in \cite{Ko12}, after applying Lemmata \ref{lema0} and \ref{lema}, we have a new expression of the complementary hypergeometric Euler numbers $\widehat E_{N,n}$.

\begin{theorem}
For $N\ge 0$ and $n\geq 1$,
\begin{multline*}
\widehat E_{N,2 n}^{(r)}
=(2 n)!\sum_{t_1 + 2t_2 + \cdots + nt_n=n}\binom{t_1+\cdots + t_n}{t_1, \dots,t_n}(-1)^{t_1+\cdots+t_n}\\
\times\left(\frac{(2 N+1)!}{(2 N+3)!}\right)^{t_1}\left(\frac{(2 N+1)!}{(2 N+5)!}\right)^{t_2}\cdots\left(\frac{(2 N+1)!}{(2 N+2 n+1)!}\right)^{t_n}\,.
\end{multline*}
Moreover,
$$
\frac{(-1)^n(2 N+1)!}{(2 N+2 n+1)!}=\begin{vmatrix} \frac{\widehat E_{N,2}}{2!} & 1 & & \\
\frac{\widehat E_{N,4}}{4!} & \ddots &  \ddots & \\
\vdots & \ddots &  \ddots & 1\\
\frac{\widehat E_{N, 2 n}}{(2 n)!} & \cdots &  \frac{\widehat E_{N,4}}{4!} & \frac{\widehat E_{N,2}}{2!} \\
 \end{vmatrix}\,,
$$
and
{\small
\begin{multline*}
\begin{pmatrix} 1 &  &  & & \\
-\frac{\widehat E_{N,2}}{2!} & 1 &   &  & \\
\frac{\widehat E_{N,4}}{4!} & -\frac{\widehat E_{N,2}}{2!}  &  1  &  & \\
\vdots &  &  \ddots &  & \\
\frac{(-1)^n\widehat E_{N,2 n}}{(2 n)!} & \cdots &  \frac{\widehat E_{N,4}}{4!} & -\frac{\widehat E_{N,2}}{2!} &1
\end{pmatrix}^{-1}\\
=\begin{pmatrix} 1 &  &  & & \\
\frac{(2 N+1)!}{(2 N+3)!} & 1 &   &  & \\
\frac{(2 N+1)!}{(2 N+5)!} & \frac{(2 N+1)!}{(2 N+3)!}  &  1  &  & \\
\vdots &  &  \ddots &  & \\
\frac{(2 N+1)!}{(2 N+2 n+1)!} & \cdots &  \frac{(2 N+1)!}{(2 N+5)!} & \frac{(2 N+1)!}{(2 N+3)!} &1
\end{pmatrix}\,.
\end{multline*}
}
\label{th2345}
\end{theorem}

When $N=0$ in Theorem \ref{th2345}, we have a different expression for the original complementary Euler numbers $\widehat E_{n}$.

\begin{Cor}
For $n\geq 1$
\begin{multline*}
\widehat E_{2 n}
=(2 n)!\sum_{t_1 + 2t_2 + \cdots + nt_n=n}\binom{t_1+\cdots + t_n}{t_1, \dots,t_n}(-1)^{t_1+\cdots+t_n}\\
\times\left(\frac{1}{3!}\right)^{t_1}\left(\frac{1}{5!}\right)^{t_2}\cdots \left(\frac{1}{(2 n+1)!}\right)^{t_n}\,.
\end{multline*}
Moreover,
$$
\frac{(-1)^n}{(2 n+1)!}=\begin{vmatrix} \frac{\widehat E_{2}}{2!} & 1 & & \\
\frac{\widehat E_{4}}{4!} & \ddots &  \ddots & \\
\vdots & \ddots &  \ddots & 1\\
\frac{\widehat E_{2 n}}{(2 n)!} & \cdots &  \frac{\widehat E_{4}}{4!} & \frac{\widehat E_{2}}{2!} \\
 \end{vmatrix}\,.
$$
\end{Cor}

\section*{Acknowledgements} The second author was partly supported by China National Science Foundation Grant (No. 11501477), the Fundamental Research Funds for the Central Universities (No. 20720170001) and the Science Fund of Fujian Province (No. 2015J01024).

\end{document}